\documentclass[a4paper, 11pt]{amsart}    
\usepackage{latexsym, amsmath, amsthm, amssymb, setspace, verbatim}
\usepackage[all]{xy}
\usepackage[utf8]{inputenc} \usepackage[T1]{fontenc} \usepackage{lmodern}
\usepackage{hyperref, aliascnt}
\usepackage{enumitem}

\title[]{Actions of certain torsion-free elementary amenable groups on strongly self-absorbing \cstar-algebras}
\author{Gábor Szabó}
\address{Department of Mathematical Sciences, University of Copenhagen, \phantom{----------}\linebreak \text{}\hspace{3.0mm} Universitetsparken 5, DK-2100 Copenhagen {\O}, Denmark.}
\email{gabor.szabo@math.ku.dk}
\thanks{Supported by the Danish National Research Foundation through the Centre for Symmetry and Deformation (DNRF92), and the European Union's Horizon 2020 research and innovation programme under the Marie Sklodowska-Curie grant agreement 746272}

\numberwithin{equation}{section}

\begin{document}

\renewcommand\matrix[1]{\left(\begin{array}{*{10}{c}} #1 \end{array}\right)}  
\newcommand\set[1]{\left\{#1\right\}}  

\newcommand{\IA}[0]{\mathbb{A}} \newcommand{\IB}[0]{\mathbb{B}}
\newcommand{\IC}[0]{\mathbb{C}} \newcommand{\ID}[0]{\mathbb{D}}
\newcommand{\IE}[0]{\mathbb{E}} \newcommand{\IF}[0]{\mathbb{F}}
\newcommand{\IG}[0]{\mathbb{G}} \newcommand{\IH}[0]{\mathbb{H}}
\newcommand{\II}[0]{\mathbb{I}} \renewcommand{\IJ}[0]{\mathbb{J}}
\newcommand{\IK}[0]{\mathbb{K}} \newcommand{\IL}[0]{\mathbb{L}}
\newcommand{\IM}[0]{\mathbb{M}} \newcommand{\IN}[0]{\mathbb{N}}
\newcommand{\IO}[0]{\mathbb{O}} \newcommand{\IP}[0]{\mathbb{P}}
\newcommand{\IQ}[0]{\mathbb{Q}} \newcommand{\IR}[0]{\mathbb{R}}
\newcommand{\IS}[0]{\mathbb{S}} \newcommand{\IT}[0]{\mathbb{T}}
\newcommand{\IU}[0]{\mathbb{U}} \newcommand{\IV}[0]{\mathbb{V}}
\newcommand{\IW}[0]{\mathbb{W}} \newcommand{\IX}[0]{\mathbb{X}}
\newcommand{\IY}[0]{\mathbb{Y}} \newcommand{\IZ}[0]{\mathbb{Z}}

\newcommand{\CA}[0]{\mathcal{A}} \newcommand{\CB}[0]{\mathcal{B}}
\newcommand{\CC}[0]{\mathcal{C}} \newcommand{\CD}[0]{\mathcal{D}}
\newcommand{\CE}[0]{\mathcal{E}} \newcommand{\CF}[0]{\mathcal{F}}
\newcommand{\CG}[0]{\mathcal{G}} \newcommand{\CH}[0]{\mathcal{H}}
\newcommand{\CI}[0]{\mathcal{I}} \newcommand{\CJ}[0]{\mathcal{J}}
\newcommand{\CK}[0]{\mathcal{K}} \newcommand{\CL}[0]{\mathcal{L}}
\newcommand{\CM}[0]{\mathcal{M}} \newcommand{\CN}[0]{\mathcal{N}}
\newcommand{\CO}[0]{\mathcal{O}} \newcommand{\CP}[0]{\mathcal{P}}
\newcommand{\CQ}[0]{\mathcal{Q}} \newcommand{\CR}[0]{\mathcal{R}}
\newcommand{\CS}[0]{\mathcal{S}} \newcommand{\CT}[0]{\mathcal{T}}
\newcommand{\CU}[0]{\mathcal{U}} \newcommand{\CV}[0]{\mathcal{V}}
\newcommand{\CW}[0]{\mathcal{W}} \newcommand{\CX}[0]{\mathcal{X}}
\newcommand{\CY}[0]{\mathcal{Y}} \newcommand{\CZ}[0]{\mathcal{Z}}

\newcommand{\FA}[0]{\mathfrak{A}} \newcommand{\FB}[0]{\mathfrak{B}}
\newcommand{\FC}[0]{\mathfrak{C}} \newcommand{\FD}[0]{\mathfrak{D}}
\newcommand{\FE}[0]{\mathfrak{E}} \newcommand{\FF}[0]{\mathfrak{F}}
\newcommand{\FG}[0]{\mathfrak{G}} \newcommand{\FH}[0]{\mathfrak{H}}
\newcommand{\FI}[0]{\mathfrak{I}} \newcommand{\FJ}[0]{\mathfrak{J}}
\newcommand{\FK}[0]{\mathfrak{K}} \newcommand{\FL}[0]{\mathfrak{L}}
\newcommand{\FM}[0]{\mathfrak{M}} \newcommand{\FN}[0]{\mathfrak{N}}
\newcommand{\FO}[0]{\mathfrak{O}} \newcommand{\FP}[0]{\mathfrak{P}}
\newcommand{\FQ}[0]{\mathfrak{Q}} \newcommand{\FR}[0]{\mathfrak{R}}
\newcommand{\FS}[0]{\mathfrak{S}} \newcommand{\FT}[0]{\mathfrak{T}}
\newcommand{\FU}[0]{\mathfrak{U}} \newcommand{\FV}[0]{\mathfrak{V}}
\newcommand{\FW}[0]{\mathfrak{W}} \newcommand{\FX}[0]{\mathfrak{X}}
\newcommand{\FY}[0]{\mathfrak{Y}} \newcommand{\FZ}[0]{\mathfrak{Z}}

\renewcommand{\phi}[0]{\varphi}
\newcommand{\eps}[0]{\varepsilon}

\newcommand{\id}[0]{\operatorname{id}}		
\newcommand{\eins}[0]{\mathbf{1}}			
\newcommand{\ad}[0]{\operatorname{Ad}}
\newcommand{\fin}[0]{{\subset\!\!\!\subset}}
\newcommand{\Aut}[0]{\operatorname{Aut}}
\newcommand{\dst}[0]{\displaystyle}
\newcommand{\cstar}[0]{\ensuremath{\mathrm{C}^*}}
\newcommand{\dist}[0]{\operatorname{dist}}
\newcommand{\cc}[0]{\simeq_{\mathrm{cc}}}
\newcommand{\scc}[0]{\simeq_{\mathrm{scc}}}
\newcommand{\vscc}[0]{\simeq_{\mathrm{vscc}}}
\newcommand{\ue}[0]{{~\approx_{\mathrm{u}}}~}
\newcommand{\GL}[0]{\operatorname{GL}}
\newcommand{\Hom}[0]{\operatorname{Hom}}
\newcommand{\dimrokc}[0]{\dim_{\mathrm{Rok}}^{\mathrm{c}}}


\newtheorem{satz}{Satz}[section]		

\newaliascnt{corCT}{satz}
\newtheorem{cor}[corCT]{Corollary}
\aliascntresetthe{corCT}
\providecommand*{\corCTautorefname}{Corollary}
\newaliascnt{lemmaCT}{satz}
\newtheorem{lemma}[lemmaCT]{Lemma}
\aliascntresetthe{lemmaCT}
\providecommand*{\lemmaCTautorefname}{Lemma}
\newaliascnt{propCT}{satz}
\newtheorem{prop}[propCT]{Proposition}
\aliascntresetthe{propCT}
\providecommand*{\propCTautorefname}{Proposition}
\newaliascnt{theoremCT}{satz}
\newtheorem{theorem}[theoremCT]{Theorem}
\aliascntresetthe{theoremCT}
\providecommand*{\theoremCTautorefname}{Theorem}
\newtheorem*{theoreme}{Theorem}

\theoremstyle{definition}

\newtheorem*{conjecturee}{Conjecture}

\newaliascnt{conjectureCT}{satz}
\newtheorem{conjecture}[conjectureCT]{Conjecture}
\aliascntresetthe{conjectureCT}
\providecommand*{\conjectureCTautorefname}{Conjecture}
\newaliascnt{defiCT}{satz}
\newtheorem{defi}[defiCT]{Definition}
\aliascntresetthe{defiCT}
\providecommand*{\defiCTautorefname}{Definition}
\newtheorem*{defie}{Definition}
\newaliascnt{notaCT}{satz}
\newtheorem{nota}[notaCT]{Notation}
\aliascntresetthe{notaCT}
\providecommand*{\notaCTautorefname}{Notation}
\newtheorem*{notae}{Notation}
\newaliascnt{remCT}{satz}
\newtheorem{rem}[remCT]{Remark}
\aliascntresetthe{remCT}
\providecommand*{\remCTautorefname}{Remark}
\newtheorem*{reme}{Remark}
\newaliascnt{exampleCT}{satz}
\newtheorem{example}[exampleCT]{Example}
\aliascntresetthe{exampleCT}
\providecommand*{\exampleCTautorefname}{Example}
\newaliascnt{questionCT}{satz}
\newtheorem{question}[questionCT]{Question}
\aliascntresetthe{questionCT}
\providecommand*{\questionCTautorefname}{Question}
\newtheorem*{questione}{Question}


\newcounter{theoremintro}
\renewcommand*{\thetheoremintro}{\Alph{theoremintro}}

\newaliascnt{theoremiCT}{theoremintro}
\newtheorem{theoremi}[theoremiCT]{Theorem}
\aliascntresetthe{theoremiCT}
\providecommand*{\theoremiCTautorefname}{Theorem}

\newaliascnt{conjectureiCT}{theoremintro}
\newtheorem{conjecturei}[conjectureiCT]{Conjecture}
\aliascntresetthe{conjectureiCT}
\providecommand*{\conjectureiCTautorefname}{Conjecture}

\newaliascnt{defiiCT}{theoremintro}
\newtheorem{defii}[defiiCT]{Definition}
\aliascntresetthe{defiiCT}
\providecommand*{\defiiCTautorefname}{Definition}


\begin{abstract} 
In this paper we consider a bootstrap class $\FC$ of countable discrete groups, which is closed under countable unions and extensions by the integers, and we study actions of such groups on \cstar-algebras.
This class includes all torsion-free abelian groups, poly-$\IZ$-groups, as well as other examples.
Using the interplay between relative Rokhlin dimension and semi-strongly self-absorbing actions established in prior work, we obtain the following two main results for any group $\Gamma\in\FC$ and any strongly self-absorbing \cstar-algebra $\CD$:
\begin{enumerate}[label=\textup{(\arabic*)},leftmargin=*]
\item There is a unique strongly outer $\Gamma$-action on $\CD$ up to (very strong) cocycle conjugacy.
\item If $\alpha: \Gamma\curvearrowright A$ is a strongly outer action on a separable, unital, nuclear, simple, $\CD$-stable \cstar-algebra with at most one trace, then it absorbs every $\Gamma$-action on $\CD$ up to (very strong) cocycle conjugacy.
\end{enumerate}
In fact we establish more general relative versions of these two results for actions of amenable groups that have a predetermined quotient in the class $\FC$.
For the monotracial case, the proof comprises an application of Matui--Sato's equivariant property (SI) as a key method.
\end{abstract}

\maketitle

\tableofcontents


\section*{Introduction}

The present work is a continuation of the work initiated in \cite{Szabo18rd}.
As such, our aim is to study \cstar-dynamical systems and to classify them up to cocycle conjugacy.
We refer to the introduction of \cite{Szabo18rd} or \cite{Szabo18kp} and the references therein for a proper motivation and historical overview of the classification theory of (discrete) group actions on von Neumann algebras and \cstar-algebras.

In short, the present technology in the realm of \cstar-algebras has not yet arrived at the point where one can reasonably attempt to classify actions of general amenable groups on all simple \cstar-algebras covered by the Elliott program; see \cite{Winter17} for an overview of recent developments on the latter.
This situation is in contrast to the well-understood situation of amenable group actions on injective factors; see for example \cite{Ocneanu85, Masuda07, Masuda13}.\footnote{There are of course many other possible references in this context, but listing them is beyond the scope of the present work.}
In order to gain some understanding about how to go about handling actions of general amenable groups on \cstar-algebras in the first place, it is beneficial (as a first step) to restrict one's attention to one of the most rigid types of \cstar-algebras, namely the strongly self-absorbing ones \cite{TomsWinter07}.
This special setup comes with a priori more angles of attack than the general case, such as the approach propagated in \cite{Szabo18ssa, Szabo18ssa2, Szabo17ssa3} to exploit strong self-absorption at the dynamical level, which already beared fruits in the context of outer actions of amenable groups on Kirchberg algebras \cite{Szabo18kp}. 
The insight from \cite{Szabo18kp} given by behavior of equivariant $KK$-theory for group actions on strongly self-absorbing \cstar-algebras gives rise to the following conjecture from the introduction of \cite{Szabo17ssa3}\footnote{The statement given here is a slightly refined version, in that we postulate not only the uniqueness for sufficiently outer actions, but also that they should all be automatically semi-strongly self-absorbing.}, which may be interpreted as an Ocneanu-type rigidity phenomenon; cf.\ \cite{Ocneanu85}.

\begin{conjecturei} \label{conjecture-a}
Let $\CD$ be a strongly self-absorbing \cstar-algebra.
Then for every countable torsion-free amenable group $\Gamma$, there is a unique strongly outer $\Gamma$-action on $\CD$ up to cocycle conjugacy, and such an action is semi-strongly self-absorbing.
\end{conjecturei}

On the one hand, we note that torsion a priori gives a $K$-theoretical obstruction to such a rigid behavior, which does not appear in the von Neumann algebraic context.
For example, it is not hard to construct non-cocycle conjugate outer actions on the Cuntz algebra $\CO_2$ for any finite abelian group; cf.\ \cite{Izumi04, Izumi04II, BarlakSzabo17}.
On the other hand, we note that an application of \cite{Jones83} gives the failure of uniqueness for actions of non-amenable groups in the monotracial case, and a much more drastic converse for non-amenable groups has been recently considered in \cite{GardellaLupini18uhf, GardellaLupini18} by Gardella--Lupini.

In the case of abelian groups and $\CD$ satisfying the UCT, \autoref{conjecture-a} has been positively solved in \cite{Szabo17ssa3}, building in a crucial way on both the structure theory of semi-strongly self-absorbing actions and prior work by Matui \cite{Matui08, Matui11} and Izumi--Matui \cite{IzumiMatui10} handling the case $\Gamma=\IZ^d$ when $\CD\ncong\CZ$.
Some special cases for $\CD=\CZ$ have been solved before by Sato \cite{Sato10} and Matui--Sato \cite{MatuiSato12, MatuiSato14}, and one should note that one of Kishimoto's early pioneering works \cite{Kishimoto95} handled the case $\Gamma=\IZ$ and $\CD$ being a UHF algebra.

In this paper, we shall furthermore confirm \autoref{conjecture-a} for the following class of groups.

\begin{defii} \label{bootstrap-definition}
We define $\FC$ to be the smallest class of groups that contains the trivial group, is closed under isomorphism, countable directed unions, and extensions by $\IZ$.
\end{defii}

It is easy to see that the class $\FC$ contains all torsion-free abelian groups, all poly-$\IZ$ groups, but also other examples such as the reduced wreath product $\IZ\wr\IZ$.
Evidently $\FC$ is contained in the class of all torsion-free elementary amenable groups, and it is perhaps less trivial that this inclusion is strict.
For example, it is known that there are torsion-free poly-cyclic groups which are not poly-$\IZ$; see \cite[page 16]{LennoxRobinson}.
For groups in the class $\FC$, our main results are as follows; see \autoref{cor:ssa-uniqueness-C} and \autoref{cor:ssa-absorption-C}, respectively.

\begin{theoremi} \label{first-theorem}
Let $\Gamma\in\FC$.
Let $\CD$ be a strongly self-absorbing \cstar-algebra.
Then any two strongly outer $\Gamma$-actions on $\CD$ are (very strongly) cocycle conjugate.
Moreover, any such action is semi-strongly self-absorbing.
\end{theoremi}

\begin{theoremi} \label{second-theorem}
Let $\Gamma\in\FC$.
Let $\CD$ be a strongly self-absorbing \cstar-algebra and $A$ a separable, unital, nuclear, simple, $\CD$-stable \cstar-algebra with at most one trace.
Let $\alpha: \Gamma\curvearrowright A$ be a strongly outer action.
Then for every action $\gamma: \Gamma\curvearrowright\CD$, the actions $\alpha$ and $\alpha\otimes\gamma$ are (very strongly) cocycle conjugate.
\end{theoremi}

The proof of \autoref{first-theorem} presented in this paper is self-contained and does not rely on any special cases treated elsewhere.
Moreover, as one might have hoped, we can dispense with the UCT assumption in our approach, because the proof relies just on strong self-absorption rather than the precise fine structure of the underlying \cstar-algebra $\CD$.\footnote{One may of course note that the UCT could possibly be redundant in this context; cf.\ \cite[Corollary 6.7]{TikuisisWhiteWinter17}. 
Due to the presently mysterious status of the UCT problem, however, it is fair to say that a proof not needing it can be regarded as more satisfactory.}
We note that the purely infinite case within our main result has significant overlap with part of Izumi--Matui's treatment of poly-$\IZ$ group actions on Kirchberg algebras \cite{IzumiMatui18}, which they have proved long before this work was initiated, even though their work has yet to be published; see also \cite{Izumi12OWR}.

The most important tool in the proof of our main results is given by the interplay between Rokhlin dimension relative to subgroups and the absorption of semi-strongly self-absorbing actions, as established in \cite{Szabo18rd}.
In essence, our proof goes by showing that, if viewed as a property of amenable groups $\Gamma$, the statements in the two theorems above are closed under extensions by $\IZ$.
Since the permanence properties from \cite{Szabo17ssa3} also show that these statements are closed under countable direct unions of groups, the main results then simply follow from the definition of $\FC$ as a bootstrap class.
In fact it follows by this approach more generally that the statement of \autoref{conjecture-a} is closed under extensions by groups in $\FC$; this is recorded in \autoref{thm:ssa-induction}.

In order to handle extensions by $\IZ$, one needs to show that all strongly outer actions as in \autoref{second-theorem} satisfy certain Rokhlin-type conditions relative to any normal subgroup $H$ such that $\Gamma/H\cong\IZ$, which allows one to apply the main result from \cite{Szabo18rd}.
This Rokhlin-type theorem is the only ingredient where the proof has to handle the case of finite and infinite \cstar-algebras separately.
The desired property follows without too much effort from \cite{Szabo18kp} and \cite{Nakamura00} in the purely infinite case, but more work is required for the monotracial case.
Among other methods, we arrange a relative version of finite Rokhlin dimension with commuting towers by embedding certain equivariant copies of prime dimension drop algebras into central sequence algebras, which exploits Matui--Sato's notion of equivariant property (SI) in a crucial way \cite{MatuiSato12, MatuiSato14, Sato17}.\footnote{We note that in a similar but slightly different context, the type of approximately central embedding technique utilized here has also been independently discovered in ongoing work of Gardella--Phillips--Wang \cite{GardellaPhillipsWang18}.}
Compared to similar Rokhlin-type theorems such as \cite{Liao16, Liao17}, the key difference is that we can arrange the resulting Rokhlin towers to commute with each other, which is both necessary for the methods of \cite{Szabo18rd} to be applicable and is in a sense a special feature of $\IZ$ as a (relative) acting group.

It may be relevant to note that the assumption about having ``at most one trace'' in \autoref{second-theorem} can be dispensed with at the cost of assuming that the group action in question has Matui--Sato's weak Rokhlin property, which is a priori more than strong outerness.
However, it is well-known by Matui--Sato's work that the weak Rokhlin property is equivalent to strong outerness in the monotracial case and in fact the proofs of our main results can be performed in this setup without having to refer to the weak Rokhlin property at all.
It is therefore a conscious decision not to go beyond the monotracial case in this paper, in favor of a more in-depth study of the connections between strong outerness and the weak Rokhlin property warranted by this remark, which shall be the subject of subsequent work.

The paper is organized as follows:
In Section \ref{sec:1}, we remind the reader about the underlying concepts of this paper such as cocycle conjugacy, (central) sequence algebras constructed from free ultrafilters, strong outerness, and property (SI).
Most of Section \ref{sec:2} is dedicated to the monotracial case of our main results, in particular showing the required relative Rokhlin-type theorem for strongly outer actions.
Once the Rokhlin-type theorem is in place, we obtain a proof of our main results in Section \ref{sec:3} as an application of various results and modified ideas from \cite{Szabo17ssa3, Szabo18rd}.


\section{Preliminaries}
\label{sec:1}

Throughout the paper, we will freely use basic techniques from the Elliott classification program for simple nuclear \cstar-algebras \cite{Rordam}, as well as their general structure theory.
In particular we assume familiarity with nuclearity \cite{BrownOzawa}, strongly self-absorbing \cstar-algebras \cite{TomsWinter07}, the Jiang--Su algebra \cite{JiangSu99}, and order zero maps \cite{WinterZacharias09}.
Moreover we refer to \cite{Szabo18ssa, Szabo18ssa2, Szabo17ssa3} for the theory of (semi-)strongly self-absorbing actions.

\begin{defi}
Let $\Gamma$ be a countable discrete group.
Let $\alpha: \Gamma\curvearrowright A$ and $\beta: \Gamma\curvearrowright B$ be two  actions on unital \cstar-algebras.
We say that $\alpha$ and $\beta$ are cocycle conjugate, written $\alpha \cc \beta$, if there exists an isomorphism $\phi: A\to B$ and an $\alpha$-cocycle $\set{w_g}_{g\in\Gamma}\subset \CU(A)$ with 
\[
\ad(w_g)\circ\alpha_g=\phi^{-1}\circ\beta_g\circ\phi \quad\text{for all } g\in\Gamma.
\]

If it is possible to choose $\phi$ and $w$ such that there exists a sequence $x_n\in\CU(A)$ with $z_g=\lim_{n\to\infty} x_n\alpha_g(x_n^*)$ for all $g\in\Gamma$, then $\alpha$ and $\beta$ are said to be strongly cocycle conjugate, written $\alpha\scc \beta$.

If it is moreover possible to choose $\phi$ and $w$ such that there exists a continuous path $x: [0,\infty)\to\CU(A)$ with $x_0=\eins$ such that $w_g=\lim_{t\to\infty} x_t\alpha_g(x_t^*)$ for all $g\in\Gamma$, then $\alpha$ and $\beta$ are said to be very strongly cocycle conjugate, written $\alpha\vscc \beta$.
\end{defi}

\begin{defi}
Fix a free ultrafilter $\omega$ on $\IN$.
Let $A$ be a \cstar-algebra.
One defines the ultrapower of $A$ as
\[
A_\omega=\ell^\infty(\IN,A)/\set{ (x_n)_n \mid \lim_{n\to\omega} \| x_n\| = 0 }.
\]
The constant sequences yield a canonical copy of $A$ in the ultrapower.
The central sequence algebra is the relative commutant
\[
A_\omega\cap A' = \set{ x\in A_\omega \mid [x,A]=0 }.
\]
\end{defi}

\begin{nota}
If $\alpha: \Gamma\curvearrowright A$ is an action of a discrete group, then we obtain the ultrapower action $\alpha_\omega: \Gamma\curvearrowright A_\omega$ via componentwise application, and in fact this restrict to an action on $A_\omega\cap A'$ as well.
If $H\subseteq\Gamma$ is a fixed subgroup, then we write $A_\omega^H$ or $(A_\omega\cap A')^H$ for the fixed point algebra with respect to $\alpha_\omega|_H$.
If $H$ is additionally normal, then we have an induced action $\Gamma/H\curvearrowright A_\omega^H$ via $(gH).a = \alpha_{\omega,g}(a)$ for all $g\in\Gamma$ and $a\in A_\omega^H$.
\end{nota}

The following concept has its origins in the work of Kirchberg \cite{Kirchberg04} and Kirchberg--Rordam \cite{KirchbergRordam14} on central sequences of \cstar-algebras; see also \cite[Definition 5.5]{Liao16}.

\begin{defi}[see {\cite[Definition 4.1]{Szabo18ssa2}}]
Let $\alpha: \Gamma\curvearrowright A$ be an action of a discrete group on a \cstar-algebra.
An $\alpha$-invariant ideal $J\subseteq A$ is called a $G$-$\sigma$-ideal, if for every separable $\alpha$-invariant \cstar-subalgebra $C\subset A$, there exists a positive contraction $e\in (J\cap C')^\alpha$ such that $ec=c=ce$ for all $c\in J\cap C$.
\end{defi}

\begin{prop}[see {\cite[Proposition 4.5]{Szabo18ssa2}}] \label{prop:sigma-ideal-lss}
Let $\alpha: \Gamma\curvearrowright A$ be an action and $J\subseteq A$ a $G$-$\sigma$-ideal.
Let $B=A/J$, $\pi: A\to B$ the quotient map, and $\beta: \Gamma\curvearrowright B$ the action induced on the quotient.
Then
\begin{enumerate}[label=\textup{(\roman*)},leftmargin=*]
\item For every separable $\alpha$-invariant \cstar-subalgebra $C\subset A$, the restriction $\pi: A\cap C'\to B\cap\pi(C)'$ is surjective.
\item For every separable $\beta$-invariant \cstar-subalgebra $D\subset B$, there is an equivariant c.p.c.\ order zero map $\psi: (D,\beta)\to (A,\alpha)$ such that $\pi\circ\psi=\id_D$.
\end{enumerate}

\end{prop}

\begin{defi}[cf.\ {\cite[Section 4]{KirchbergRordam14}}] \label{def:limit-trace}
Now assume that $A$ is unital and has a unique tracial state $\tau$.
We define the limit trace $\tau^\omega$ on $A_\omega$ via
\[
\tau^\omega\big( [(x_n)_n] \big) = \lim_{n\to\omega} \tau(x_n).
\]
We define $\| x\|_{p,\omega}=\tau^\omega(| x|^p)^{1/p}$ for all $p\in [1,\infty)$.
The trace-kernel ideal in $A_\omega$ is given by
\[
\CJ_A = \set{ x\in A_\omega  \mid \| x\|_{p,\omega}=0 \text{ for all (or some) } p }.
\]
The tracial ultrapower is the quotient $A^\omega=A_\omega/\CJ_A$.
By Kaplansky's density theorem, it is easy to see (cf.\ \cite[Theorem 3.3]{KirchbergRordam14}) that for the weak closure $M=\pi_\tau(A)''$, the canonical inclusion $A^\omega\subseteq M^\omega$ into the von Neumann algebraic tracial ultrapower is in fact an isomorphism, which restrict to an isomorphism $A^\omega\cap A'\cong M^\omega\cap M'$.
\end{defi}

\begin{nota}
In the above situation, if $\alpha: \Gamma\curvearrowright A$ is any action of a countable discrete group, then the trace-kernel ideal $\CJ_A\subset A_\omega$ is $\alpha_\omega$-invariant.
Thus the componentwise application of $\alpha$ gives rise to an action $\alpha^\omega: \Gamma\curvearrowright A^\omega$ on the tracial ultrapower.
\end{nota}

The following is always true regardless of the structure of the tracial simplex of $A$, but we will stick to the monotracial case as it is enough for our present purpose.

\begin{prop} \label{prop:trace-kernel-sigma}
Let $\alpha: \Gamma\curvearrowright A$ be an action of a countable discrete group on a unital monotracial \cstar-algebra. 
Then the trace-kernel ideal $\CJ_A\subset A_\omega$ is a $G$-$\sigma$-ideal with respect to the ultrapower action $\alpha_\omega$.
\end{prop}
\begin{proof}
The proof follows almost verbatim as in \cite[Proposition 4.6]{KirchbergRordam14} by applying the so-called $\eps$-test.
For this one only needs to know that $\CJ_A$ admits an approximate unit consisting of approximately $\alpha_\omega$-invariant elements quasicentral relative to $A_\omega$, which is a general fact \cite[Proposition 1.4]{Kasparov88} due to Kasparov.
We omit the details.
\end{proof}

\begin{defi}
Let $A$ be a unital \cstar-algebra with $T(A)\neq\emptyset$.
An automorphism $\alpha$ on $A$ is called strongly outer, if it is outer, and moreover for every $\alpha$-invariant trace $\tau\in T(A)$, the induced automorphism of $\alpha$ on the weak closure $\pi_\tau(A)''$ is outer.\footnote{In particular, ``strongly outer'' is defined as ``outer'' when $A$ is traceless.}

If $\alpha: \Gamma\curvearrowright A$ is an action of a discrete group, then we say that it is (pointwise) strongly outer, if $\alpha_g$ is strongly outer whenever $g\neq 1$.
\end{defi}

\begin{rem}
In the above definition, if $A$ has a unique trace $\tau$, then an action $\alpha$ is strongly outer precisely when the action induced on the weak closure $\pi_\tau(A)''$ is an outer action in the sense of von Neumann algebras.
When $A$ is nuclear and infinite-dimensional, then it is known that $\pi_\tau(A)''$ yields the hyperfinite II$_1$-factor $\CR$ (cf.\ \cite{Connes76}), which we will use frequently.
\end{rem}

\begin{defi}[cf.\ {\cite[Proposition 4.5]{MatuiSato14} and \cite[Proposition 5.1]{Sato17}}]
Let $\Gamma$ be a countable discrete group.
Let $\alpha:\Gamma\curvearrowright A$ be an action on a separable, simple, unital, monotracial \cstar-algebra.
We say that $A$ has equivariant property (SI) relative to $\alpha$, if the following holds:

Given two positive contractions $e, f\in (A_\omega\cap A')^{\alpha_\omega}$ satisfying
\[
\|e\|_{1,\omega}=0,\quad \inf_{k\in\IN} \|\eins-f^k\|_{1,\omega} < 1,
\]
there exists a contraction $s\in (A_\omega\cap A')^{\alpha_\omega}$ such that
\[
e=s^*s \quad\text{and}\quad fs=s
\]
In particular, in case $\Gamma=\set{1}$ we say that $A$ has property (SI).
\end{defi}

\begin{theorem} \label{thm:MS-SI}
Let $A$ be a separable, simple, nuclear, unital, monotracial \cstar-algebra.
Suppose that $A$ has strict comparison.
Then for every action $\alpha: \Gamma\curvearrowright A$ of a discrete amenable group, $A$ has property (SI) relative to $\alpha$.
Moreover, $\alpha$ is equivariantly $\CZ$-absorbing.
\end{theorem}
\begin{proof}
For $\Gamma=\set{1}$, this is a special case of a well-known major insight given by Matui--Sato's work \cite[Section 4]{MatuiSato12acta}. 

For the general case, combine Propositions 4.4 and 5.1 from \cite{Sato17}.
The last part of the claim is a special case of \cite[Theorem 5.2]{Sato17} together with \cite[Corollary 3.8]{Szabo18ssa}; see also \cite[Theorem 4.9]{MatuiSato14} for a previous similar theorem.
\end{proof}


\section{Relative Rokhlin-type conditions}
\label{sec:2}

\begin{defi}[cf.\ {\cite[Definition 4.1]{Szabo18rd}}] \label{def:rel-Rokhlin-type-conditions}
Let $\Gamma$ be a countable discrete group.
Let $H\subset\Gamma$ be a normal subgroup such that $\Gamma/H\cong\IZ$, and let $g_0\in\Gamma$ be an element generating the quotient.
Let $\alpha:\Gamma\curvearrowright A$ be an action on a separable unital \cstar-algebra.
\begin{enumerate}[label=\textup{(\roman*)},leftmargin=*]
\item The Rokhlin dimension of $\alpha$ with commuting towers relative to $H$, denoted $\dimrokc(\alpha, H)$, is the smallest natural number $d\geq 0$ such that the following holds.
For every $n\geq 1$, there exist equivariant c.p.c.\ order zero maps
\[
\phi^{(0)},\dots,\phi^{(d)}: \big( \CC(\IZ/n\IZ), \IZ\textup{-shift}\big) \to \big( (A_\omega\cap A')^{H}, \alpha_{\omega,g_0} \big)
\]
with pairwise commuting ranges such that
\[
\phi^{(0)}(\eins)+\dots+\phi^{(d)}(\eins)=\eins.
\]
\item We say that $\alpha$ has the Rokhlin property relative to $H$, if for every $n\geq 1$, there exist projections $p,q\in (A_\omega\cap A')^{H}$ such that
\[
\eins = \sum_{i=0}^{n-1} \alpha_{\omega,g_0}^i(p)+\sum_{j=0}^n \alpha_{\omega,g_0}^j(q).
\]
\end{enumerate}
\end{defi}

\begin{prop} \label{prop:Rp-vs-dimrokc}
Suppose that $\alpha: \Gamma\curvearrowright A$ is an action of a countable discrete group on a separable unital \cstar-algebra.
Let $H\subset\Gamma$ be a normal subgroup such that $\Gamma/H\cong\IZ$.
If $\alpha$ has the Rokhlin property relative to $H$, then $\dimrokc(\alpha,H)\leq 1$.
\end{prop}
\begin{proof}
This is completely analogous to \cite[Proposition 2.8]{HirshbergWinterZacharias15}.
\end{proof}

The key feature of Rokhlin dimension with commuting towers comes from the following theorem, which is a special case of  \cite[Theorem 4.4]{Szabo18rd}.

\begin{theorem} \label{thm:dimrok-absorption}
Let $\Gamma$ be a countable, discrete group and $H\subset\Gamma$ be a normal subgroup such that $\Gamma/H\cong\IZ$.
Let $\alpha:\Gamma\curvearrowright A$ be an action on a separable unital \cstar-algebra, and let $\gamma: \Gamma\curvearrowright\CD$ be a semi-strongly self-absorbing, unitarily regular action.
If $\dimrokc(\alpha,H)<\infty$ and $\alpha|_H\cc(\alpha\otimes\gamma)|_H$, then $\alpha\cc\alpha\otimes\gamma$.
\end{theorem}

\begin{defi}
For given numbers $p,q\in\IN$, recall the dimension drop algebra
\[
Z_{p,q} = \set{ f\in\CC\big( [0,1], M_p\otimes M_q\big) \mid f(0)\in M_p\otimes\eins, f(1)\in\eins\otimes M_q }.
\]
If $p$ and $q$ are relatively prime, this is called a prime dimension drop algebra.
\end{defi}

\begin{defi}[cf.\ {\cite{RordamWinter10}}] \label{def:ZU}
Let $k\geq 2$ be a natural number.
One defines $Z_{k,k+1}^U$ to be the universal (unital) \cstar-algebra generated by the range of a c.p.c.\ order zero map $\psi^U: M_k\to Z_{k,k+1}^U$ and a contraction $s_U\in Z_{k,k+1}^U$ subject to the relations
\[
s_U^*s_U = \eins-\psi(\eins),\quad \psi(e_{1,1})s_U=s_U.
\]
\end{defi}

Recall the following characterization of the isomorphism $Z_{k,k+1}^U\cong Z_{k,k+1}$:

\begin{theorem}[cf.\ {\cite[Proposition 5.1]{RordamWinter10} and \cite[Section 2]{Sato10}}] \label{thm:RW+S}
Let $k\geq 2$ be a natural number.
Suppose that $u: [0,1]\to\CU(M_k\otimes M_k)$ is any unitary path satisfying
\[
u_0=\eins,\quad u_1=\sum_{i,j=1}^k e_{i,j}^{(k)}\otimes e_{j,i}^{(k)}.
\]
Consider the element in $M_k\otimes M_{k+1}$ given by
\[
v=\sum_{j=1}^k e_{1,j}^{(k)}\otimes e_{j,k+1}^{(k+1)}.
\]
Consider the functions
\[
w, s \in Z_{k,k+1}
\]
and the map
\[
\psi: M_k \to \CC\big( [0,1], M_k\otimes M_{k+1} \big)
\]
given by the formulas
\[
w(t) = u_t\oplus \cos(\pi t/2)\cdot \eins^{(k)}\otimes e^{(k+1)}_{k+1,k+1};
\]
\[
\psi(x) = w(t)( x\otimes\eins^{(k+1)} )w(t)^*;
\]
\[
s(t) = \sin(\pi t/2) w(t) v.
\]
Then $\psi$ is a well-defined c.p.c.\ order zero map $\psi: M_k\to Z_{k,k+1}$ so that $\psi$ and the element $s$ satisfy the relations in \autoref{def:ZU}.
Moreover, the resulting $*$-homomorphism $\Phi: Z_{k,k+1}^U\to Z_{k,k+1}$ is an isomorphism.
\end{theorem}

\begin{rem} \label{rem:abstract-vs-concrete-model}
In the above construction, the specific isomorphism $Z_{k,k+1}^U\cong Z_{k,k+1}$ depends on the choice of the unitary path $u$.
Denote by $D_k\subset M_k\otimes M_k$ the \cstar-subalgebra generated by all elementary tensors of the form $z\otimes z$ for $z\in M_k$.
In particular, we may find a unitary path $u$ taking values in the relative commutant $(M_k\otimes M_k)\cap D_k'$.

Suppose that $G$ is some group and $\nu: G\to M_{k-1}$ is a unitary representation.
Then we consider $\eins\oplus\nu: G\to M_k$ as a unitary representation so that each group element acts as a unit on $e_{1,1}$.
For notational convenience, set $\mu_g=(\eins\oplus\nu_g)\otimes (\eins\oplus\nu_g\oplus\eins) \in M_k\otimes M_{k+1}$, which defines yet another unitary representation.

We may consider the unique action $\delta^{U,\nu}: G\curvearrowright Z_{k,k+1}^U$ given by $\delta^{U,\nu}(s_U)=s_U$ and $\delta^{U,\nu}\circ\psi^U = \psi^U\circ\ad(\eins\oplus\nu_g)$.
On the other hand, we may consider the action $\delta^\nu: G\curvearrowright Z_{k,k+1}$ given by $\delta^\nu_g(f)=\ad(\mu_g)(f)$.\footnote{Note that $\mu_g$ is not a unitary in $Z_{k,k+1}$, yet conjugation with this unitary still induces a well-defined automorphism on $Z_{k,k+1}$.}
Then $\delta^{U,\nu}$ and $\delta^\nu$ are conjugate.
\end{rem}
\begin{proof}
As stated above, we consider a continuous path $u: [0,1]\to M_k\otimes M_k$ that pointwise commutes with $D_k$ and satisfies $u_0=\eins$ and $u_1=\sum_{i,j=1}^k e_{i,j}\otimes e_{j,i}$.
Construct the elements $w, c_j, s\in Z_{k,k+1}$ and let $\Phi: Z_{k,k+1}^U \to Z_{k,k+1}$ be the isomorphism induced by them as stated in \autoref{thm:RW+S}.
We will show that $\Phi$ is automatically equivariant with respect to the $G$-actions $\delta^{U,\nu}$ and $\delta^\nu$.

Since any unitary of the form $u_t$ commutes with $\nu_g\otimes\nu_g$ for all $g\in G$, it follows from the definition of $w\in Z_{k,k+1}\subset \CC\big( [0,1], M_k\otimes M_{k+1} \big)$ that it must commute with $\mu_g$ for all $g\in G$.
In particular, $w$ is in the fixed point algebra of $\delta^\nu$.

We compute
\[
\begin{array}{ccl}
w(1)v &=& \dst \Big( \sum_{i,l=1}^k e^{(k)}_{i,l}\otimes e^{(k+1)}_{l,i} + \eins^{(k)}\otimes e^{(k+1)}_{k+1,k+1} \Big)\Big( \sum_{j=1}^k e_{1,j}^{(k)}\otimes e_{j,k+1}^{(k+1)} \Big) \\
&=& \dst \sum_{j=1}^k e^{(k)}_{j,j} \otimes e^{(k+1)}_{1,k+1} \ = \ \eins^{(k)}\otimes e^{(k+1)}_{1,k+1}.
\end{array}
\]
By definition of the unitaries $\mu_g$, we see that this element commutes with them.
Since we have observed earlier that $w(t)$ commutes with $\mu_g$ for all $t\in [0,1]$, it follows that also $v$ and $s$ commutes with $\mu_g$ for $g\in G$.
In other words, $s$ is also in the fixed point algebra of $\delta^\nu$.

Lastly, we compute for every $x\in M_k$ that
\[
\begin{array}{ccl}
\big[ \delta^\nu_g\circ\psi(x) \big](t) &=& \mu_gw(t)(x\otimes\eins^{(k+1)})w(t)^*\mu_g^* \\
&=& w(t)\mu_g(x\otimes\eins^{(k+1)})\mu_g^*w(t)^* \\
&=& w(t)(\ad(\eins\oplus\nu_g)(x)\otimes\eins^{(k+1)})w(t)^* \\
&=& \big[ \psi\circ\ad(\eins\oplus\nu_g)(x) \big](t).
\end{array}
\]
By the definition of the isomorphism $\Phi$ and the action $\delta^{U,\nu}$, this means $\delta^\nu_g\circ\Phi\circ\psi^U = \Phi\circ\delta^{U,\nu}_g\circ\psi^U$.
Since $s$ is in the fixed point algebra of $\delta^\nu$ and $s_U\in Z_{k,k+1}^U$ is in the fixed point algebra of $\delta^{U,\nu}$, it follows that $\Phi$ is equivariant. 
This finishes the proof.
\end{proof}

\begin{rem} \label{rem:left-regular-character}
Let $k\geq 2$ be a natural number and $G$ a discrete group. If $\lambda: G\to M_k$ is a unitary representation that arises from a permutation $\sigma: G\curvearrowright\set{1,\dots,k}$ via $\lambda_g(e_i)=e_{\sigma(i)}$, then there is a unitary representation $\nu: G\to M_{k-1}$ such that $\lambda$ is unitarily conjugate to a $\eins\oplus\nu$.
\end{rem}
\begin{proof}
Since $\lambda$ arises from a permutation as given in the statement, it follows that the vector
\[
x = k^{-1/2} (1,\dots,1)\in\IC^k
\]
is a unit vector fixed by $\lambda$.
Then $\lambda$ restricts to a unitary representation on the space $x^\perp\cong\IC^{k-1}$.
Thus the claim follows easily from here.
\end{proof}

\begin{lemma} \label{lem:existence}
Let $A$ be a separable, unital, simple, nuclear, monotracial $\CZ$-stable \cstar-algebra.
Let $\alpha: \Gamma\curvearrowright A$ be a strongly outer action of a countable discrete amenable group.
Let $k\geq 2$ and $\nu: \Gamma\to M_{k-1}$ a unitary represenation.
Then there exists a unital equivariant $*$-homomorphism from $(Z_{k,k+1}^U,\delta^{U,\nu})$ to $(A_\omega\cap A',\alpha_\omega)$.
\end{lemma}
\begin{proof}
We consider $(B,\beta)=(M_k,\ad(\eins\oplus\nu))$ as a monotracial \cstar-dynamical system.
Denote by $\tau$ the unique trace on $A$.
As $\alpha$ is strongly outer, the induced action on the weak closure $\pi_\tau(A)''$ is outer.
As $\pi_\tau(A)''\cong\CR$, it follows from Ocneanu's theorem \cite{Ocneanu85} that this action is cocycle conjugate to
\[
\alpha\otimes\ad(\eins\oplus\nu)^{\otimes\infty}: \Gamma\curvearrowright\CR\bar{\otimes} M_k^{\bar{\otimes}\infty}.\footnote{Here the symbol $\bar{\otimes}$ denotes the spatial tensor product of von Neumann algebras.}
\]
Thus can find a unital equivariant $*$-homomorphism
\[
\kappa: (B,\beta)\to (A^\omega\cap A',\alpha^\omega).
\]
By \autoref{prop:trace-kernel-sigma} and \autoref{prop:sigma-ideal-lss}, we find an equivariant c.p.c.\ order zero lift 
\[
\psi: (B,\beta)\to (A_\omega\cap A',\alpha_\omega)
\]
such that $\|\eins-\psi(\eins)\|_{1,\omega}=0$. 
We notice that $\psi(e_{1,1})\in (A_\omega\cap A')^{\alpha_\omega}$ and $\tau_\omega(\psi(e_{1,1})^m)=\frac1k$ for all $m\geq 1$.

Since $A$ has equivariant property (SI) relative to $\alpha$ by \autoref{thm:MS-SI}, we can find a contraction $s\in (A_\omega\cap A')^{\alpha_\omega}$ such that $s^*s=\eins-\psi(\eins)$ and $\psi(e_{1,1})s=s$.
We see that the pair $(\psi,s)$ satisfies the universal property used to define $Z_{k,k+1}^U$, and thus we obtain a unique unital $*$-homomorphism $\phi: Z_{k,k+1}^U \to A_\omega\cap A'$ with $\phi\circ\psi^U=\psi$ and $\phi(s_U)=s$.
As $s$ is fixed by $\alpha_\omega$ and $\psi$ was equivariant by choice, we see that $\phi$ becomes equivariant with respect to $\delta^{U,\nu}$ and $\alpha_\omega$.
\end{proof}

\begin{nota} \label{nota:model-action}
Henceforth, we will denote by $\lambda^{(N)}\in M_N$ the unitary that is induced by the left-regular representation of $\IZ_N$ on $\IC^N \cong L^2(\IZ_N)$.
More specifically, one has
\[
\lambda^{(N)} = e_{1,N}+\sum_{k=1}^{N-1} e_{k+1,k} = \matrix{ 0 & \dots & \dots & 1 \\ 1 & 0 & \dots & 0 \\  & 1 && \vdots \\  \vdots && \ddots & \\ 0 & \dots && 1}.
\]
We will also denote by $\delta^N$ the automorphism on $Z_{N,N+1}$ induced by $\ad\big( \lambda^{(N)}\otimes (\lambda^{(N)}\oplus\eins)\big)$.
\end{nota}

\begin{lemma}[see {\cite[Lemma 2.2]{Kishimoto96}}] \label{lem:Kishimoto-lemma}
Let $n\in\IN$ and $\eps> 0$ be given.
Then there exists $K\in\IN$ such that for every $N\geq K$ there exist projections $p_0,\dots,p_{n-1},q_0,\dots,q_n\in M_{N+1}$ such that
\[
\eins = \sum_{j=0}^{n-1} p_j + \sum_{l=0}^n q_l
\]
and
\[
p_{j+1} =_\eps \sigma^N(p_j) \ (\mathrm{mod}\ n),\quad q_{l+1} =_\eps \sigma^N(q_l) \ (\mathrm{mod}\ n+1),
\]
where $\sigma^N=\ad(\lambda^{(N)}\oplus\eins)$.
\end{lemma}

\begin{lemma} \label{lem:dimrokc-2-in-ZN}
Let $n\in\IN$ and $\eps> 0$ be given.
Then there exists $K\in\IN$ such that for every $N\geq K$ with $n|N$, there exist pairwise commuting positive elements
\[
a_0,\dots,a_{n-1},b_0,\dots,b_{n-1},c_0,\dots,c_n \in Z_{N,N+1}
\]
such that
\begin{itemize}
\item $\eins=\sum_{j=0}^{n-1} a_j + b_j + \sum_{k=0}^n c_k$;
\item each of the collections $\set{a_j}_{j=0}^{n-1}$, $\set{b_j}_{j=0}^{n-1}$ and $\set{c_l}_{l=0}^{n}$ consists of pairwise orthogonal elements;
\item $\delta^N(a_j)=_\eps a_{j+1}$ \textup{mod} $n$;
\item $\delta^N(b_j)=_\eps b_{j+1}$ \textup{mod} $n$;
\item $\delta^N(c_l)=_\eps c_{l+1}$ \textup{mod} $n+1$;
\item $b_j\perp c_l$ for all $j=0,\dots,n-1$ and $l=0,\dots,n$.
\end{itemize}
\end{lemma}
\begin{proof}
The number $K$ is the same as the one in \autoref{lem:Kishimoto-lemma}.
Let $N\geq K$ be such that $N=n_0n$ for some $n_0\in\IN$.

Applying \autoref{lem:Kishimoto-lemma}, we find projections $p_0,\dots,p_{n-1}$ and $q_0,\dots,q_n$ with the stated properties. For $j=0,\dots,n-1$ and $l=0,\dots,n$, we define functions via
\[
b_j(t) = t\cdot \eins^{(N)}\otimes p_j,\quad c_l(t)=t\cdot \eins^{(N)}\otimes q_l,\quad t\in [0,1].
\]
This yields pairwise orthogonal elements in $Z_{N,N+1}$, and by our choice of $p_j,q_l$ they are pairwise orthogonal and satisfy $\delta^N(b_j)=_\eps b_{j+1} \ (\mathrm{mod}\ n)$ and $\delta^N(c_l)=_\eps c_{l+1} \ (\mathrm{mod}\ n+1)$.
The sum over all $b_j$ and $c_l$ equals the element given by the function $[ t\mapsto t\cdot\eins ]$.

Lastly, for $j=0,\dots,n-1$ we set
\[
a_j(t) = (1-t)\cdot \sum_{l=0}^{n_0-1} e_{1+j+ln,1+j+ln}\otimes\eins^{(N+1)} \in M_N\otimes M_{N+1},\quad t\in [0,1].
\]
This defines pairwise orthogonal functions in $Z_{N,N+1}$ satisfying $\delta^N(a_j)=a_{j+1} \ (\mathrm{mod}\ n)$, and moreover their sum is equal to the function $[t\mapsto (1-t)\eins]$.

Evidently, all of these functions constructed so far commute with each other.
Moreover their sum is equal to the unit, which shows our claim.
\end{proof}

\begin{lemma} \label{lem:dd-dimrok2}
Let $n\in\IN$ and $\eps> 0$ be given.
Then there exists $K\in\IN$ such that for every $N\geq K$ with $n|N$, there exist pairwise commuting positive elements
\[
a_0,\dots,a_{n-1},b_0,\dots,b_{n-1},c_0,\dots,c_{n-1} \in Z_{N,N+1}
\]
such that
\begin{itemize}
\item $\eins=\sum_{j=0}^{n-1} a_j + b_j + c_j$;
\item each of the collections $\set{a_j}_{j=0}^{n-1}$, $\set{b_j}_{j=0}^{n-1}$ and $\set{c_j}_{j=0}^{n-1}$ consists of pairwise orthogonal elements;
\item $\delta^N(a_j)=_\eps a_{j+1}$ \textup{mod} $n$;
\item $\delta^N(b_j)=_\eps b_{j+1}$ \textup{mod} $n$;
\item $\delta^N(c_j)=_\eps c_{j+1}$ \textup{mod} $n$.
\end{itemize}
\end{lemma}
\begin{proof}
Using \autoref{lem:dimrokc-2-in-ZN}, we can obtain such elements in exactly the same fashion as in the proof of \cite[Proposition 2.8]{HirshbergWinterZacharias15}.
\end{proof}

\begin{theorem} \label{thm:dimrok2}
Let $A$ be a separable, unital, simple, nuclear, monotracial $\CZ$-stable \cstar-algebra.
Suppose that $\alpha: \Gamma\curvearrowright A$ is a strongly outer action of an amenable group.
Let $H\subset\Gamma$ be a normal subgroup with $\Gamma/H\cong\IZ$.
Then $\dimrokc(\alpha, H)\leq 2$.
\end{theorem}
\begin{proof}
Let $g_0\in\Gamma$ be an element generating the quotient.
Due to the $\eps$-test \cite[Lemma 3.1]{KirchbergRordam14}, it is enough to show for a fixed $\eps>0$ and $n\geq 1$ that there exist pairwise commuting positive contractions $a_j, b_j, c_j\in (A_\omega\cap A')^H$ such that:
\begin{itemize}
\item $\eins=\sum_{j=0}^{n-1} a_j + b_j + c_j$;
\item each of the collections $\set{a_j}_{j=0}^{n-1}$, $\set{b_j}_{j=0}^{n-1}$ and $\set{c_j}_{j=0}^{n-1}$ consists of pairwise orthogonal elements;
\item $\alpha_{\omega,g_0}(a_j)=_\eps a_{j+1}$ \textup{mod} $n$;
\item $\alpha_{\omega,g_0}(b_j)=_\eps b_{j+1}$ \textup{mod} $n$;
\item $\alpha_{\omega,g_0}(c_j)=_\eps c_{j+1}$ \textup{mod} $n$.
\end{itemize}
For the pair $(\eps,n)$, choose $N\geq 1$ big enough to satisfy the conclusion of \autoref{lem:dd-dimrok2}.
By our assumptions on $g_0,\Gamma,H$, we get a well-defined action $\gamma: \Gamma\curvearrowright Z_{N,N+1}$ via $\gamma|_H = \id$ and $\gamma_{g_0}=\delta^N$ in the sense of \autoref{nota:model-action}.
By \autoref{rem:left-regular-character} and \autoref{rem:abstract-vs-concrete-model}, $\gamma$ is conjugate to $\delta^{U,\nu}: \Gamma\curvearrowright Z^U_{N,N+1}$ for some unitary representation $\nu: \Gamma\to M_{N-1}$.
Thus \autoref{lem:existence} allows us to find a unital equivariant $*$-homomorphism
\[
\phi: (Z_{N,N+1},\gamma) \to (A_\omega\cap A',\alpha_\omega).
\]
By our definition of $\gamma$, this can also be viewed as a unital equivariant $*$-homomorphism
\[
\phi: (Z_{N,N+1},\delta^{N}) \to ( (A_\omega\cap A')^H,\alpha_{\omega,g_0}).
\]
Hence the desired elements exist by \autoref{lem:dd-dimrok2}.
\end{proof}

\begin{example}[cf.\ \cite{Kishimoto95}] \label{ex:UHF-Rp}
Every UHF algebra $\IU$ of infinite type admits a strongly self-absorbing automorphism with the Rokhlin property.
\end{example}
\begin{proof}
For a fixed $n\in\IN$, we consider the direct sum $M_n\oplus M_{n+1}$, and observe that the unitary
$s_n = \lambda^{(n)}\oplus\lambda^{(n+1)}$
defines an inner automorphism for which the (standard) minimal projections $(e_{1,1}\oplus 0) \in M_n\oplus 0$ and $(0\oplus e_{1,1})\in 0\oplus M_{n+1}$ generate a Rokhlin multitower of length $n$ on the nose.

Now let $\IU$ be a UHF algebra of infinite type.
Clearly there exists a unital $*$-homomorphism $\iota^{(n)}: M_n\oplus M_{n+1}\to\IU$.
We define
\[
\alpha = \bigotimes_{n\in\IN} \ad(\iota^{(n)}(s_n))^{\otimes\infty} : \IZ \curvearrowright (\IU^{\otimes\infty})^{\otimes\infty} \cong \IU
\]
which will satisfy the Rokhlin property by construction.
It is also strongly self-absorbing by \cite[Proposition 5.2]{Szabo18ssa2}.
\end{proof}

\begin{theorem} \label{thm:UHF-abs-Rp}
Let $A$ be a separable, unital, simple, nuclear, monotracial \cstar-algebra.
Suppose that $\alpha: \Gamma\curvearrowright A$ is a strongly outer action of an amenable group.
Let $H\subset\Gamma$ be a normal subgroup with $\Gamma/H\cong\IZ$.
Then for any UHF algebra $\IU$ of infinite type, the action $\alpha\otimes\id_\IU$ has the Rokhlin property relative to $H$.
\end{theorem}
\begin{proof}
Certainly we may assume $\alpha\cc\alpha\otimes\id_\IU$ without loss of generality.
Let $g_0\in\Gamma$ be an element generating the quotient.
Let $\psi\in\Aut(\IU)$ be a strongly self-absorbing automorphism with the Rokhlin property, as in \autoref{ex:UHF-Rp}.
By the assumption on $g_0,\Gamma,H$, we obtain a well-defined action $\gamma:\Gamma\curvearrowright\IU$ via $\gamma|_H=\id$ and $\gamma_{g_0}=\psi$.
By replacing $\gamma$ if necessary\footnote{This is actually not necessary by \autoref{thm:dimrok-absorption}.}, we may also assume $\gamma\cc\gamma\otimes\id_\IU$.
Hence $\gamma$ is unitarily regular by \cite[Proposition 2.19]{Szabo18ssa2}.
Evidently $\gamma$ is strongly self-absorbing and has the Rokhlin property relative to $H$.

Clearly we have $(\alpha\otimes\gamma)|_H = (\alpha\otimes\id_\IU)|_H \cong \alpha|_H$.
Since we know that $\dimrokc(\alpha,H)\leq 2$ from \autoref{thm:dimrok2}, we may apply \autoref{thm:dimrok-absorption} to deduce $\alpha\cc\alpha\otimes\gamma$.
Hence $\alpha$ also has the Rokhlin property relative to $H$.
\end{proof}

\begin{rem}
With a further reduction argument, it is possible to improve the conclusion of \autoref{thm:UHF-abs-Rp} to include arbitrary infinite-dimensional UHF algebras in place of $\IU$.
Since we do not need this level of generality to obtain our main results, this shall not be pursued here.
\end{rem}

Lastly, let us also consider the purely infinite case to have a unified proof of the main result within the next section:

\begin{theorem} \label{thm:pi-Rp}
Let $A$ be a Kirchberg algebra.
Let $\alpha: \Gamma\curvearrowright A$ be a pointwise outer action of an amenable group.
Let $H\subset\Gamma$ be a normal subgroup with $\Gamma/H\cong\IZ$.
Then $\alpha$ has the Rokhlin property relative to $H$.
\end{theorem}
\begin{proof}
Let $g_0\in\Gamma$ be an element generating the quotient.
Let $u\in\CU(\CO_\infty)$ be a unitary with full spectrum $\IT$. 
Then by the properties of $g_0,\Gamma,H$, we may associate a unique unitary representation $w: \Gamma\to\CU(\CO_\infty)$ via $w|_H=\eins$ and $w_{g_0}=u$.
Moreover we get a well-defined action
\[
\gamma=\ad(w)^{\otimes\infty}: \Gamma\curvearrowright\CO_\infty^{\otimes\infty}\cong\CO_\infty.
\]
Evidently $\gamma|_H=\id$ and $\gamma_{g_0}$ is an aperiodic automorphism.
Since $\gamma_{g_0}$ has the Rokhlin property by \cite[Theorem 1]{Nakamura00}, it follows by definition that $\gamma$ has the Rokhlin property relative to $H$.
Moreover, it follows from \cite[Theorem 3.5]{Szabo18kp} that $\alpha\cc\alpha\otimes\gamma$.
This shows our claim. 
\end{proof}


\section{Actions on strongly self-absorbing \cstar-algebras}
\label{sec:3}

For what follows recall \autoref{bootstrap-definition} of the bootstrap class of groups $\FC$ from the introduction.

\begin{theorem} \label{thm:absorption-induction}
Let $\CD$ be a strongly self-absorbing \cstar-algebra and let $A$ be a separable, unital, simple, nuclear, $\CD$-stable \cstar-algebra with at most one trace.
Let $\alpha: \Gamma\curvearrowright A$ be a strongly outer action of an amenable group.
Suppose that $H\subset\Gamma$ is a normal subgroup such that $\Gamma/H\in\FC$.
Let $\gamma: \Gamma\curvearrowright\CD$ be a semi-strongly self-absorbing action.
If $\alpha|_H \cc (\alpha\otimes\gamma)|_H$, then $\alpha\cc\alpha\otimes\gamma$.
\end{theorem}
\begin{proof}
Let $\FF$ be the class of all countable amenable groups $\Lambda$ such that the conclusion of this theorem holds whenever one has $\Gamma/H\cong\Lambda$ instead of $\Gamma/H\in\FC$.
Evidently the trivial group is in $\FF$, and $\FF$ is closed under extensions.
Moreover it follows directly from \cite[Theorem 5.6(ii)]{Szabo17ssa3} that $\FF$ is closed under countable directed unions.
By the definition of the class $\FC$, it suffices to show $\IZ\in\FF$ in order to obtain $\FC\subseteq\FF$, which will prove the claim.

So let us assume $\Gamma/H\cong\IZ$.
If $A$ is finite, then it follows from \autoref{thm:dimrok2} that $\dimrokc(\alpha,H)\leq 2$.
If $A$ is infinite, then it follows from \autoref{thm:pi-Rp} and \autoref{prop:Rp-vs-dimrokc} that $\dimrokc(\alpha,H)\leq 1$.
So in all cases we have $\dimrokc(\alpha,H)\leq 2$.
Note that $\gamma$ is equivariantly $\CZ$-stable by either \autoref{thm:MS-SI} or \cite[Theorem 3.4]{Szabo18kp} (depending on whether $\CD$ is finite or infinite), so in particular it is unitarily regular; see \cite[Proposition 2.19]{Szabo18ssa2}.
Thus if $\alpha|_H \cc (\alpha\otimes\gamma)|_H$, then $\alpha\cc\alpha\otimes\gamma$ follows by \autoref{thm:dimrok-absorption} and the proof is complete.
\end{proof}

\begin{theorem} \label{thm:ssa-induction}
Let $\CD$ be a strongly self-absorbing \cstar-algebra.
Let $\gamma,\gamma^{(1)},\gamma^{(2)}: \Gamma\curvearrowright\CD$ be strongly outer actions of an amenable group.
Suppose that $H\subset\Gamma$ is a normal subgroup such that $\Gamma/H\in\FC$.
\begin{enumerate}[label=\textup{(\roman*)},leftmargin=*]
\item If $\gamma|_H$ is semi-strongly self-absorbing, then so is $\gamma$. \label{thm:ssa-induction:1}
\item If $\gamma^{(i)}|_H$ is semi-strongly self-absorbing for $i=1,2$ and $\gamma^{(1)}|_H\cc\gamma^{(2)}|_H$, then $\gamma^{(1)}\cc\gamma^{(2)}$. \label{thm:ssa-induction:2}
\end{enumerate}
In particular, the statement of \autoref{conjecture-a} is closed under extensions by groups in the class $\FC$.
\end{theorem}
\begin{proof}
First we observe that if \ref{thm:ssa-induction:1} is true, then \ref{thm:ssa-induction:2} follows directly from it together with \autoref{thm:absorption-induction}.
So we need to show \ref{thm:ssa-induction:1}.

Similarly as in the proof of \autoref{thm:absorption-induction}, let us consider the class $\FF$ of all countable amenable groups $\Lambda$ such that \ref{thm:ssa-induction:1} holds whenever $\Gamma/H\cong\Lambda$ instead of $\Gamma/H\in\FC$.
Evidently the trivial group is in $\FF$ and $\FF$ is closed under extensions.
Moreover it follows directly from \cite[Theorem 5.6(i)]{Szabo17ssa3} that $\FF$ is closed under countable directed unions.
By the definition of the class $\FC$, it suffices to show $\IZ\in\FF$ in order to obtain $\FC\subseteq\FF$, which will prove the claim.

So let us assume $\Gamma/H\cong\IZ$.
Let $g_0\in\Gamma$ be an element generating the quotient.

\textbf{Step 1:} It follows from either \autoref{thm:MS-SI} or \cite[Theorem 3.4]{Szabo18kp} (depending on whether $\CD$ is finite or infinite) that $\gamma\cc\gamma\otimes\id_\CZ$.\footnote{Note that $\CD\cong\CD\otimes\CZ$ is known due to \cite{Winter11}.}
By virtue of \cite[Theorem 6.6]{Szabo17ssa3}, the claim reduces to the special case where $\gamma\cc\gamma\otimes\id_\IU$ for some UHF algebra $\IU$ of infinite type.
So let us make this assumption from now on.

\textbf{Step 2:}
We assume $\gamma|_H$ is semi-strongly self-absorbing.
We claim that $\gamma$ has approximately $\Gamma$-inner flip.\footnote{This part of the proof will be similar to \cite[Theorem 6.6]{Szabo18rd}.}

Set $B=\CD\otimes\CD$ and $\beta=\gamma\otimes\gamma: \Gamma\curvearrowright B$.
Denote by $\Sigma$ the flip automorphism on $B$, which is $\beta$-equivariant.
Since $\gamma|_H$ is semi-strongly self-absorbing by assumption, the flip is approximately $H$-inner.
By \cite[Proposition 3.6]{Szabo18ssa2}, we find unitaries $x,y\in B_\omega^H$ such that $\ad(xyx^*y^*)(b)=\Sigma(b)$ for all $b\in B$.
We set $u=xyx^*y^*$ and observe that $u$ is homotopic to the unit inside $\CU\big( B_\omega^H )$ by \cite[Proposition 2.19]{Szabo18ssa2}.

On the other hand, we also have
\[
\beta_{\omega,g_0}^n(u)b\beta_{\omega,g_0}^n(u)^* = \beta_{\omega,g_0}^n\big( u\beta_{g_0}^{-n}(b)u^* \big) = \beta_{g_0}^n\circ\Sigma\circ\beta_{g_0}^{-n}(b) = \Sigma(b)
\]
for all $b\in B$ and $n\in\IZ$.
Hence we have $w_n := u\beta_{\omega,g_0}^n(u)^*\in (B_\omega\cap B')^H$ for all $n$.
Clearly $\set{w_n}_{n\in\IZ}$ is the $\beta_{\omega,g_0}$-cocycle over $\IZ$ associated to the unitary $w_1=u\beta_{\omega,g_0}(u)^*$.

Since $u$ is homotopic to the unit in $\CU\big(B_\omega^H\big)$, this is possible with some $L$-Lipschitz unitary path for some $L> 0$.
Since $\beta|_H$ is semi-strongly self-absorbing, it follows from \cite[Lemma 3.12]{Szabo18ssa2} that all of the unitaries $w_n$ are homotopic to the unit inside $\CU\big( (B_\omega\cap B')^H\big)$ via a $2L$-Lipschitz unitary path.
Let $D$ be some separable, $\beta_{\omega,g_0}$-invariant \cstar-subalgebra of $(B_\omega\cap B')^H$ containing the cocycle $\set{w_n}_{n\in\IZ}$ along with all such unitary paths for each $n\in\IZ$.

Since we have assumed $\gamma\cc\gamma\otimes\id_\IU$, we also have $\beta\cc\beta\otimes\id_\IU$, and therefore by \autoref{thm:UHF-abs-Rp} (if $\CD$ is finite) or \autoref{thm:pi-Rp} (if $\CD$ is infinite) the action $\beta$ has the Rokhlin property relative to $H$.
So for any $n\in\IN$ we have projections $p,q\in (B_\omega\cap B')^H$ such that $\eins=\sum_{j=0}^{n-1} \beta_{\omega,g_0}^j(p)+\sum_{l=0}^n \beta_{\omega,g_0}^l(q)$.
By a standard reindexation trick, we may additionally assume $[p,D]=0=[q,D]$.

This allows us to employ the same argument as in the proof of \cite[Proposition 4.3]{Kishimoto98II} to deduce that there exists a unitary $v\in (B_\omega\cap B')^H$ with $u\beta_{\omega,g_0}(u)^*=v\beta_{\omega,g_0}(v)^*$; see also \cite{HermanOcneanu84}.
Set $z=v^*u$.
Then $z$ is evidently a unitary in $B_\omega^H$, but it also satisfies $z=\beta_{\omega,g_0}(z)$, hence in fact $z\in B_\omega^{\beta_\omega}$.
Moreover we have
\[
zbz=v^*ubu^*v = v^*\Sigma(b)v = \Sigma(b) \quad\text{for all } b\in B.
\]
This shows that the flip automorphism $\Sigma$ is indeed approximately $\Gamma$-inner.

\textbf{Step 3:}
From ``Step 2'' above it follows that $\gamma^{\otimes\infty}: \Gamma\curvearrowright\CD^{\otimes\infty}$ is a strongly self-absorbing action; see \cite[Proposition 3.3]{Szabo18ssa}.
By our assumption that $\gamma|_H$ is semi-strongly self-absorbing, we have $\gamma|_H \cc (\gamma\otimes\gamma^{\otimes\infty})|_H$.
By applying \autoref{thm:absorption-induction} we see that $\gamma\cc\gamma\otimes\gamma^{\otimes\infty}\cc\gamma^{\otimes\infty}$, which shows that $\gamma$ is indeed semi-strongly self-absorbing.
This completes the proof.
\end{proof}

\begin{cor} \label{cor:ssa-uniqueness-C}
Let $\Gamma\in\FC$.
Let $\CD$ be a strongly self-absorbing \cstar-algebra.
Then up to (very strong) cocycle conjugacy, there exists a unique strongly outer $\Gamma$-action on $\CD$.
\end{cor}
\begin{proof}
Uniqueness up to cocycle conjugacy follows from \autoref{thm:ssa-induction}\ref{thm:ssa-induction:2} for $H=\set{1}$.
The uniqueness up to very strong cocycle conjugacy is due to the fact that these actions are all semi-strongly self-absorbing by \autoref{thm:ssa-induction}\ref{thm:ssa-induction:1}, and hence one may apply the strengthened McDuff-type result \cite[Theorem 3.2]{Szabo17ssa3}.
\end{proof}

\begin{cor} \label{cor:ssa-absorption-C}
Let $\Gamma\in\FC$.
Let $\CD$ be a strongly self-absorbing \cstar-algebra and $A$ a separable, unital, simple, nuclear, $\CD$-stable \cstar-algebra with at most one trace.
Let $\alpha: \Gamma\curvearrowright A$ be an action.
Then $\alpha$ is strongly outer if and only if $\alpha\cc\alpha\otimes\gamma$ for every action $\gamma: \Gamma\curvearrowright\CD$.
\end{cor}
\begin{proof}
Let $\gamma^0: \Gamma\curvearrowright\CD$ be any strongly outer action, such as the non-commutative Bernoulli shift on $\bigotimes_\Gamma\CD\cong\CD$.
By \autoref{thm:ssa-induction}\ref{thm:ssa-induction:1} applied to $H=\set{1}$, $\gamma_0$ is a semi-strongly self-absorbing action.
Clearly any $\gamma^0$-absorbing action is also strongly outer, so this shows the ``if'' part.

For the ``only if'' part, assume that $\alpha$ is strongly outer.
By \autoref{thm:absorption-induction} applied to $H=\set{1}$, it follows that $\alpha\cc\alpha\otimes\gamma^0$.
In fact one has $\alpha\vscc\alpha\otimes\gamma^0$ due to the strengthened McDuff-type result \cite[Theorem 3.2]{Szabo17ssa3}.
Moreover, it follows from \autoref{cor:ssa-uniqueness-C} that $\gamma^0\vscc\gamma^0\otimes\gamma$ for every action $\gamma: \Gamma\curvearrowright\CD$, so indeed one always has $\alpha\vscc\alpha\otimes\gamma$.
This finishes the proof.
\end{proof}


\bibliographystyle{gabor}
\bibliography{master}

\end{document}